\documentclass[11pt]{article}

\usepackage{latexsym,color,graphicx,amsmath,amssymb, amsthm}

\usepackage{enumitem}

\def\R{{\mathbb R}}

\def\Q{{\mathbb Q}}
\def\C{{\mathbb C}}

\def\x{{\bf x}}
\def\y{{\bf y}}

\newtheorem{theorem}{Theorem}[section]
\newtheorem{lemma}[theorem]{Lemma}

\newtheorem{corollary}[theorem]{Corollary}

\theoremstyle{definition}
\newtheorem{defi}[theorem]{Definition}
 
\begin{document}

\title{A note on distinct distances\thanks{Work on this paper was supported by Grant 892/13 from the Israel Science Foundation,
by the Israeli Centers of Research Excellence (I-CORE) program (Center No.~4/11), and by a Shulamit Aloni Fellowship from the Israeli Ministry of Science.}}

\author{
Orit E. Raz\thanks{%
Hebrew University of Jerusalem, Jerusalem, Israel.
{\sl oritraz@mail.huji.ac.il} }}

\maketitle

\begin{abstract}
We show that, for a constant-degree algebraic 
curve $\gamma$ in $\R^D$,
every set of $n$ points on $\gamma$ spans at least
$\Omega(n^{4/3})$ distinct distances, unless $\gamma$ is an {\it algebraic helix}, in the sense of Charalambides~\cite{Char1}. 
This improves the earlier bound $\Omega(n^{5/4})$ of 
Charalambides~\cite{Char1}.

We also show that,
for every set $P$ of $n$ points that 
lie on a $d$-dimensional constant-degree algebraic variety $V$ in $\R^D$,
there exists a subset $S\subset P$ of size at least $\Omega(n^{\frac{4}{9+12(d-1)}})$, 
such that
$S$ spans $\binom{|S|}{2}$ distinct distances.
This improves the earlier bound of 
$\Omega(n^{\frac{1}{3d}})$ of Conlon, Fox, Gasarch, Harris, 
Ulrich, and Zbarsky~\cite{CF}.

Both results are consequences of a common technical tool. 
\end{abstract}


\noindent {\bf AMS classification.} 52C10, 05D10.

\maketitle

\section{Introduction}

In this paper we study two mildly related problems that involve distinct distances in a point set.
The unifying theme of these problems is that they are both based on a common technical tool (Theorem~\ref{lem:quad} below).

\paragraph{Distinct distances on a curve.}
The distinct distances problem of Erd\H os~\cite{Erd} asks for the minimum number of distinct distances spanned by 
any set $P$ of $n$ points in the plane.
The $\sqrt n\times\sqrt n$ integer grid in the plane induces $\Theta (n/\sqrt{\log n})$ distinct distances, and Erd\H os conjectured
that this number is asymptotically tight. 
In a recent breakthrough, Guth and Katz~\cite{GK2} proved that every set of $n$ points in the plane 
spans at least $\Omega(n/\log n)$ distinct distances, which almost matches Erd\H os's upper bound.

An instance of the problem suggested by Purdy (e.g., see \cite[Section 5.5]{BMP})
asks for the minimum number of distinct distances spanned by
pairs of $P_1\times P_2$, where, for each $i=1,2$, $P_i$ is a set of $n$ points that lie on a line $\ell_i$.
Elekes and R\'onyai \cite{ER00} showed that, in contrast with the general case, this number is at least $\omega(n)$, unless
the lines $\ell_1, \ell_2$ are either orthogonal or parallel to one another 
(where in the latter cases sets with only $O(n)$ distinct distances between them can then be constructed).
Sharir, Sheffer, and Solymosi~\cite{SSS} strengthened the result by showing that 
the number of distinct distances spanned by $P_1\times P_2$ in the non-parallel, non-orthogonal case 
is at least $\Omega(n^{4/3})$.
This result was later generalized by Pach and De Zeeuw~\cite{PdZ} to 
the case where, for $i=1,2$, $P_i$ is a set of points that lie on
some irreducible constant-degree
algebraic curve $\gamma_i$ in the plane. 
They showed that in this case the number of distinct distances spanned by 
$P_1\times P_2$ is again at least $\Omega(n^{4/3})$, 
unless $\gamma_1,\gamma_2$ is either a pair of orthogonal lines, a pair of (possibly coinciding) parallel lines, or a pair of (possibly coinciding)
concentric circles.

The first to consider the distinct distances problem (in the plane and in higher dimensions), 
with points restricted to an arbitrary constant-degree algebraic curve, 
was Charalambides~\cite{Char1}. He showed that if a set $P$ of $n$ points lies on a constant-degree algebraic curve $\gamma$
in $\R^D$, for any $D\ge 2$, then the number of distinct distances spanned by $P$ is at least $\Omega(n^{5/4})$, 
unless $\gamma$ is an {\it algebraic helix}, defined as follows (see Charalambides~\cite[Definition 1.5 and Lemma 7.4]{Char1}).
\begin{defi}\label{helixdef}
An {\it algebraic helix} is an irreducible algebraic curve $\gamma\subset\R^D$ which is either a straight line, or has a
parameterization of the form
\begin{equation}\label{eq:alghel}
\gamma(t)=(a_1 \cos(\lambda_1 t), a_1\sin(\lambda_1 t),\ldots, a_k \cos(\lambda_k t), a_k\sin(\lambda_k t))\in\R^{2k}, 
\end{equation}
for some embedding of $\R^{2k}$ in $\R^D$, with $k\le D/2$, and where all the ratios $\lambda_j/\lambda_i$ are rational, for $i,j=1,\ldots,k$. 
\end{defi}

In the plane an algebraic helix is just a line or a circle, so the result of 
Pach and De Zeeuw provides a generalization (to the bipartite case)
and an improved bound of Charalmbides' result for the case $D=2$. 

The first main result of this paper is to show that the bound 
$\Omega(n^{5/4})$ in \cite{Char1} can be replaced by 
$\Omega(n^{4/3})$ also for $D>2$, essentially by combining the general result 
of Raz, Sharir, and De Zeeuw~\cite{RSdZ} with the analysis in \cite{Char1}.
More precisely, we have the following theorem.

\begin{theorem}\label{charimp}
Let $\gamma$ be an irreducible constant-degree algebraic curve in $\R^D$, for any $D\ge 3$. 
Then every set $P$ of $n$ points on $\gamma$ spans at least $\Omega(n^{4/3})$ distinct distances, with
a constant of proportionality that depends only on the degree of $\gamma$ (and is independent of $D$),
unless $\gamma$ is an algebraic helix.
\end{theorem}

The proof of Theorem~\ref{charimp} is given in Section~\ref{sec:char}.

\paragraph{Subsets with all-distinct distances.}
A related problem of Erd\H os \cite{Erd} asks for $h_d(n)$, the maximum $t$ such that every
set $P$ of $n$ points in $\R^d$ contains a subset $S$ of $t$ points such that all $\binom{t}{2}$
distances between the pairs of
points in $S$ are distinct. Erd\H os conjectured that $h_1(n) = (1+o(1))\sqrt{n}$. 
The set $P = \{1,\ldots, n\}$ gives
the upper bound $h_1(n)=O (\sqrt{n})$, 
while a lower bound of the form $h_1(n) = \Omega(\sqrt{n})$ follows from
a result of Koml\'os, Sulyok, and Szemer\'edi~\cite{KSS}
(see Section~\ref{sec:KSS} for a proof of this fact and more details).
In two dimensions, utilizing an important estimate from the work of Guth and Katz \cite{GK2}, Charalambides~\cite{Char2} 
proved that $h_2(n) = \Omega((n/\log n)^{1/3})$. 
The $\sqrt{n}\times\sqrt{n}$ grid has $O(n/\sqrt{\log n})$ 
distinct distances and it follows that $h_2(n) = O(n^{1/2}/(\log n)^{1/4})$. 
In higher dimensions, Thiele \cite{Thi} showed that $h_d(n) = \Omega(n^{1/(3d-2)})$, and this was recently improved 
by Conlon, Fox, Gasarch, Harris, 
Ulrich, and Zbarsky~\cite{CF} to 
$h_d(n) = \Omega(n^{1/(3d-3)}(\log n)^{1/3-2/(3d-3)})$.

In \cite{CF} the authors investigated the more general function $h_{a,d}(n)$,  
the largest integer $t$ such that any set of $n$ points in $\R^d$ contains a subset of $t$ points for which all the
non-zero $(a-1)$-dimensional volumes of the $\binom{t}{a}$ subsets of size $a$ are distinct.
Note that $h_{2,d}(n)=h_d(n)$.
They showed that $h_{a,d}(n)=\Omega(n^{\frac{1}{(2a-1)d}})$ for all (constant) $a$ and $d$.
In addition, and as a tool for bounding $h_{a,d}(n)$,  they introduced a more general notion $h_a(V,n)$, 
for $V\subset\R^D$ a $d$-dimensional\footnote{By a $d$-dimensional variety $V\subset\R^D$ 
we mean here that $V=V_\C\cap \R^D$, where $V_\C$ is a $d$-dimensional algebraic variety in $\C^D$ which is the 
zero set of a system of exactly $D-d$ polynomials of real coefficients.}
irreducible variety,  which is the largest integer $t$ such that any set of $n$ points in $V$ contains a subset of $t$ points for which all the
non-zero $(a-1)$-dimensional volumes of the $\binom{t}{a}$ subsets of size $a$ are distinct.
They then consider the quantity
$h_{a,d,r}(n):=\min_V  h_a(V,n)$,
where the minimum ranges over all $d$-dimensional
irreducible varieties $V$ of degree $r$.
It was proved in \cite{CF} that 
$
h_{a,d,r}=\Omega\left(n^{\frac{1}{(2a-1)d}}\right),
$
with a constant of proportionality that depends on $a$, $d$, and $r$.
For the special case $a=2$, namely, the case of distinct distances, the bound is 
$h_{2,d,r}=\Omega\left(n^{\frac{1}{3d}}\right)$.

The second main result of this paper is the following improvement to the bound, as just stated,
on the quantity $h_{2,d,r}(n)$.
\begin{theorem}\label{main}
For all integers $d,r\ge 1$, we have 
$$
h_{2,d,r}(n)=\Omega\left(n^{\frac4{9+12(d-1)}}\right),
$$
where the constant of proportionality depends on $d$ and $r$.
\end{theorem}

The proof of Theorem~\ref{main} is given in Section~\ref{sec:subsets}.

\paragraph{The common technical core.}
As already noted, there is a technical core behind Theorems~\ref{charimp} and \ref{main}, which is an application of a 
recent result of Raz, Sharir, and De Zeeuw~\cite{RSdZ} (which is a strengthened version of the Elekes-Szab\'o 
theorem \cite{ES12}). More precisely, for our purposes we need a somewhat stronger version of the result 
of \cite{RSdZ}, that we establish in Theorem~\ref{lem:quad}, which is the main technical tool used for our proofs. 

Roughly speaking, Theorem~\ref{lem:quad} says that, in the context of distances\footnote{In fact, instead of 
distances one may consider any other constant-degree polynomial function over $(\R^2)^2$.} between points
that lie on some constant-degree irreducible algebraic curve, there are two dichotomic  types of curves:
The first type is of curves 
that locally behave like a line, in the sense
that, by choosing the right parameterization, the distance between 
a pair of points $p=\gamma(t)$, $q=\gamma(s)$ on the curve, is given as a function of
the difference $s-t$ of the parameters representing $p$ and $q$.
An example for a curve of this kind is a circle, say, $x^2+y^2=1$, in $\R^2$. 
Fixing some small arc of the circle, the distance between a pair of points $p=e^{it}$ and $q=e^{is}$
is determined by $|t-s|$ (namely, $\|p-q\|=2\sin(|t-s|/2)$).

The second kind of curves are those that are ``very different'' from a line. 
One property that distinguishes such curves from lines is given in Theorem~\ref{lem:quad}(i).
As a consequence of our results, one can specify other properties that distinguish curves $\gamma$ of the 
latter kind from a line. 
For example, 
no triangle with vertices supported by $\gamma$ can be moved along $\gamma$ while preserving its
edge lengths
(a posteriori this follows from
the results of Charalambides~\cite{Char1}, 
but we had to deduce this fact independently in order to show 
that this indeed characterizes curves of the second kind).

The difference between Theorem~\ref{lem:quad} and the result in \cite{RSdZ} (see Lemma~\ref{RSdZ} for the relevant statement),
is that the latter result, adapted to our context, is restricted to the ``bipartite case", where one places 
a set of points $P$ on some small arc of a curve $\gamma$, and another set $Q$ on some other 
small arc on $\gamma$ and consider distances between pairs of points $(p,q)\in P\times Q$.
Theorem~\ref{lem:quad} allows one to consider all pairwise distances spanned by a set $P$.

We view the results in this paper as a new type of applications of the Elekes-Szab\'o theorem.
We believe our approach, as well as Theorem~\ref{lem:quad}, will be useful in future applications of this theorem.

\section{Preliminary results}\label{sec:pre}
\subsection{A theorem of Koml\'os, Sulyok, and Szemer\'edi}\label{sec:KSS}
A set $A=\{a_1,\ldots, a_n\}$ of positive integers is called $B_2$ (see \cite{ET}) if the sums
$a_i+a_j$ are all distinct. Let $\Phi(n)$ denote the maximal size of a $B_2$ set consisting of positive integers not exceeding $n$.
Erd\H os and Tur\'an~\cite{ET} have shown that $\Phi(n)=\Theta(n^{1/2})$. Koml\'os, Sulyok, and Szemer\'edi~\cite{KSS} have 
shown that every set of $n$ positive integers (not necessarily $\{1,\ldots,n\}$)
contains a subset of size at least $c\Phi(n)$ which is $B_2$, for some constant $c>0$.
Combining these results we have:
\begin{theorem}[\cite{ET, KSS}]\label{KSS}
Let $X$ be a set of $n$ positive integers. 
Then there exists a subset $Y\subset X$ of size $\Omega(n^{1/2})$ 
which is $B_2$.
\end{theorem}

The result in \cite{KSS} can be extended to sets of positive real numbers (not necessarily integers), a 
fact which is also mentioned in \cite{CF}. Since we use it in our analysis, and for completeness, we provide a proof of this fact; I would like to thank an anonymous referee for showing me this simple reduction.
\begin{lemma}\label{kss}
Let $X$ be a set of $n$ positive real numbers. 
Then there exists a subset $Y\subset X$ of size $\Omega(n^{1/2})$ 
which is $B_2$.
\end{lemma}
\begin{proof}
Consider the vector space $V$ over $\Q$ spanned by the elements of $X$, and let $B\subset V$ form a basis for this vector space; clearly, $V$ is finite dimensional. Write $B=\{b_1,\ldots,b_k\}$ and consider the map $b_i\mapsto t^{i-1}$, which embeds $V$ (and thus $X$ and $X+X$) to the linear vector space of polynomials in $\Q[t]$ of degree at most $k-1$. Concretely, each $v\in V$ is associated with a polynomial $p_v(t)$ such that $p_v(t)=p_{v'}(t)$ iff $v=v'$ and  $p_{v+v'}(t)=p_v(t)+p_{v'}(t)$. Fixing some integer $N$ sufficiently large and applying Theorem~\ref{KSS} to $\{p_x(N)\mid x\in X\}$ proves the lemma. 
\end{proof}

\subsection{Distances between points lying on an algebraic curve}
The results in this section essentially follow from the analysis in Raz, Sharir, and De Zeeuw~\cite{RSdZ}.
The main new observation here is in our formulation of Lemma~\ref{lem:special}.\footnote{The 
results in \cite{RSdZ} characterize bivariate functions with certain properties as having the form
$f(x,y)=h(\varphi(x)+\psi(y)),$ over some domains $x\in I$ and $y\in J$. Here 
it is important for us to have $x$ and $y$ ranging over the same domain $I$, shared by both.}

\begin{lemma}[Raz, Sharir, and De Zeeuw~\cite{RSdZ}]\label{RSdZ}
Let $F\in\R[x,y,z]$ be a constant-degree irreducible polynomial, and assume that none of the derivatives 
$F_x$, $F_y$, $F_z$ is identically zero. Then
one of the following two statements holds.\\
(I) For all $A,B\subset \R$, with $|A|= |B|= n$, we have
$$
|\{(a,a',b,b')\in A\times A\times B\times B\mid \exists c\in\R.~ F(a,b,c)=F(a',b',c)=0\}|=O(n^{8/3}).
$$
(II) There exists a one-dimensional subvariety $Z_0\subset Z(F)$, such that for all $v \in Z(F)\setminus Z_0$,
there exist open intervals $I_1,I_2,I_3 \subset \R$ and one-to-one real-analytic functions $\varphi_i: I_i \to \R$
 with analytic inverses, for $i = 1, 2, 3$, such that $v\in I_1\times I_2\times I_3$ and for all $(x, y, z) \in I_1\times I_2\times I_3$
 $$
(x,y,z)\in Z(F)~~\text{ if and only if}~~~ \varphi_1(x) + \varphi_2(y) + \varphi_3(z) = 0
$$
\end{lemma}

Let $\gamma$ be a constant-degree irreducible algebraic curve in $\R^D$, and let 
$\alpha(t)=(t,\alpha_2(t),\ldots,\alpha_D(t))$, $t\in (0,1)$,
be a real-analytic parametrization of some (relatively open connected) arc  $\alpha \subset\gamma$. 
Define 
$$
\rho(x,y):=(x-y)^2+\sum_{i=2}^{D}(\alpha_i(x)-\alpha_i(y))^2,
$$
which is the squared distance between the two points on $\alpha$
parameterized by $x$ and $y$.
Let $\rho_i$ denote the derivative of the function $\rho$ with respect to 
its $i$th variable, for $i=1,2$.
Consider the transformation $T:(0,1)^4\to\R^4$, given by
$$
T(x,x',y,y'):=(\rho(x,y), \rho(x,y'),\rho(x',y),\rho(x',y')).
$$
Let $J_T$ stand for the Jacobian matrix of $T$.


\begin{lemma}\label{lem:special}
Let $\gamma, \alpha, \rho$, and $T$ be as above. Assume that $\det J_T=0$ over $(0,1)^4$.
Then there exists an open sub-interval $I\subset (0,1)$, such that, for every $x,y\in I$, 
$$\rho(x,y)=h(\varphi(x)-\varphi(y)),$$
where $\varphi,h$ are some univariate invertible analytic functions defined 
over $I$ and $J:=\varphi(I)-\varphi(I)$, respectively.
\end{lemma}

For the proof we need the following technical lemma.
\begin{lemma}\label{lem:deriv}
Let $\gamma, \alpha$, and $\rho$ be as above. Then either $\gamma$ is a line, or
there exist $a',b'\in (0,1)$ and a finite subset $I_0\subset(0,1)$ 
(of size depending on the degree of $\gamma$),  such that
$$
\rho_2(x, b')\rho_2(a', y)\rho_1(a', b')\rho_1(x, b')\rho_1(a', y)\rho_2(a',b')\neq 0,
$$
for every $x,y\in (0,1)\setminus I_0$.
\end{lemma}
\begin{proof}
Assume, without loss of generality, that $D$ is minimal, i.e., 
that $\gamma$ is not contained in any hyperplane in $\R^D$.
If $D=1$, then  $\gamma$ is a line, and we are done.
Otherwise, we have
$$
\frac12 \rho_1(x,y)=x-y+\sum_{i=2}^D\alpha_i'(x)(\alpha_i(x)-\alpha_i(y)),
$$
and
$$
-\frac12 \rho_2(x,y)=x-y+\sum_{i=2}^D\alpha_i'(y)(\alpha_i(x)-\alpha_i(y)).
$$

Fix any $x=a'\in (0,1)$. Then the zero set 
$\{\alpha(y)\mid y\in(0,1),~\rho_1(a',y)=0\}$
is contained in the hyperplane $H$ given by
$$
a'-\xi_1+\sum_{i=2}^D\alpha_i'(a')(\alpha_i(a')-\xi_i)=0.
$$ 
By our assumption, $\gamma$ is not contained in $H$, and hence must intersect it
in at most a constant number of points.

Similarly, the zero set 
$$
H':=\{\alpha(y)\mid y\in(0,1),~\rho_2(a',y)=0\}
$$
is given by
\begin{equation}\label{sphere}
a'-y+\sum_{i=2}^D\alpha_i'(y)(\alpha_i(a')-\alpha_i(y))=0,
\end{equation}
Note that if \eqref{sphere} holds identically for every $y$ in some open subinterval of $J\subset (0,1)$,
then $\rho(a',y)=c$, for some constant $c>0$ and for every $y\in (0,1)$.
This would imply that $\gamma$ is contained in the $(d-1)$-dimensional sphere of radius $c$ centered at 
$\alpha(a')$ (recall that $\gamma$ is irreducible). 
However, since we assume $\alpha(a')\in \gamma$, this leads to a contradiction. Thus $H'$ does not contain any portion
of $\alpha$.

Note also that, since $\gamma$ is an algebraic curve, the squared distance between two points in $\R^D$ is a polynomial function
in the coordinates of $(\R^{D})^2$, 
and using the implicit function theorem to obtain a polynomial expression for the derivative $\rho_2$, 
$H'$ is contained in some constant-degree (depending on the degree of $\gamma$) irreducible algebraic variety $V$.
Since $H'$ does not contain any portion
of $\alpha$, it must intersect it in at most a constant number of points (that depends on the degree of $\gamma$).

Define $I(a')\subset (0,1)$ to be the finite set of parameters representing $\alpha\cap(H\cup H')$, if any exist.

Next, fix some $y=b'\in (0,1)\setminus I(a')$. Then 
$$
\rho_1(a',b')\rho_2(a',b')\neq 0,
$$ 
and, applying a symmetric argument to the one given above, there exists a finite set $I(b')\subset (0,1)$ such that 
$$
\rho_1(x,b')\rho_2(x,b')\neq 0,
$$ 
for every $x\in (0,1)\setminus I(b')$.
Letting $I_0:=I(a')\cup I(b')$, this  completes the proof of the lemma.
\end{proof}

\begin{proof}[Proof of Lemma~\ref{lem:special}.]
If $\gamma$ is a line the assertion is trivial. 
We may therefore assume this is not the case.
By assumption, we have
\begin{equation}\label{det0}
\rho_1(x, y)\rho_2(x, y')\rho_2(x', y)\rho_1(x', y') = \rho_2(x, y)\rho_1(x, y')\rho_1(x', y)\rho_2(x',y'),
\end{equation}
for every $(x,x',y,y')\in (0,1)^4$.

By Lemma~\ref{lem:deriv}, there exist $a',b'\in (0,1)$ and an open interval $I\subset (0,1)$, such that,
for every $(x,y)\in I^2$,
\begin{equation}\label{zero}
\rho_2(x, b')\rho_2(a', y)\rho_1(a', b')\rho_1(x, b')\rho_1(a', y)\rho_2(a',b')\neq 0
\end{equation}
and, in view of \eqref{det0},
\begin{equation}\label{det}
\rho_1(x, y)\rho_2(x, b')\rho_2(a', y)\rho_1(a', b') = \rho_2(x, y)\rho_1(x, b')\rho_1(a', y)\rho_2(a',b'),
\end{equation}
for every $(x,y)\in I^2$.

Rearranging \eqref{det}, we have
\begin{equation}\label{ode}
\frac{\rho_1(x, y)}{\rho_2(x, y)}=\frac{p(x)}{q(y)},
\end{equation}
where 
$$
p(x):=\frac{\rho_1(x, b')\rho_2(a',b')}{\rho_2(x, b')}~~\text{and}~~q(y):=\frac{\rho_2(a', y)\rho_1(a', b') }{\rho_1(a', y)},
$$
and each of them is well defined and nonzero on $I$.
We consider the real-analytic primitives $\varphi,\psi$ so that $\varphi'(x)=p(x)$ on $I$ and $\psi'(y)=q(y)$ on $I$.
Since, by construction, $\varphi',\psi'$ are nonzero, the inverse mapping theorem implies that 
each of $\varphi,\psi$ has an analytic inverse on its image.

We repeat the analysis in \cite[Lemma 3.17]{RSdZ} to 
show that the differential equation \eqref{ode} imposes a
restrictive form on $\rho(x,y)$.
Express the function $\rho(x,y)$ in terms of new coordinates $(\xi,\eta)$, given by
\begin{equation}\label{sys}
\xi=\varphi(x)+\psi(y),\quad \eta=\varphi(x)-\psi(y).
\end{equation}
Since each of $\varphi$, $\psi$ 
is an injection in $I$, the system (\ref{sys}) is invertible in $I^2$.
Returning to the standard notation, denoting partial derivatives by variable subscripts, we have
$$
\xi_{x}=\varphi'(x), ~~~
\xi_{y}
=\psi'(y), ~~~
\eta_{x}
=\varphi'(x),~~~\text{and}~~~
\eta_{y}
=-\psi'(y).
$$
Using the chain rule, we obtain
\[
\rho_{1}=\rho_{\xi} \xi_x+\rho_{\eta} \eta_x=\varphi'(x)(\rho_{\xi}+\rho_{\eta})=p(x)(\rho_{\xi}+\rho_{\eta})
\]
\[
\rho_{2}=\rho_{\xi}\xi_y+\rho_{\eta}\eta_y=\psi'(y)(\rho_{\xi}-\rho_{\eta})=q(y)(\rho_{\xi}-\rho_{\eta}),
\]
which gives
$$\frac{\rho_{1}(x,y)}{p(x)}-\frac{\rho_{2}(x,y)}{q(y)}\equiv2\rho_{\eta}(x,y),
$$
on $I^2$. Combining this with \eqref{ode}, we get
$$
\rho_{\eta}(x,y)\equiv 0.
$$
This means that $\rho$ depends (locally in $I^2$) only on the variable $\xi$, so it has the form
$$
\rho(x,y)=h(\varphi(x)+\psi(y)),
$$
for a suitable analytic function $h$. 
The analyticity of $h$ is an easy consequence of the analyticity of $\varphi, \psi$, and $\rho$, 
and the fact that $\varphi'(x)$ and $\psi'(y)$ are nonzero,
combined with repeated applications of the chain rule (see also \cite{RSdZ}).
Let 
$$J:=\{\varphi(x)+\psi(y)\mid (x,y)\in I^2\}~\quad\text{and}\quad T:=\{h(z)\mid z\in J\}.$$
We observe that 
$$
\rho_{1}(x,y)=h'(\varphi(x)+\psi(y))\cdot \varphi'(x).
$$
As argued above, we have $\rho_{1}(x,y)\neq 0$ for all $(x,y)\in I^2$, 
implying that $h'(\varphi(x)+\psi(y))$ is nonzero for $(x,y)\in I^2$.
Therefore, by the inverse mapping theorem, $h:J\to T$ is invertible.
In particular, the equation $h(c)=0$ has a unique solution $c_0$ over $J$
($c_0$ exists since $\rho(x,x)=0$ for each $x\in I$, implying that $0\in T$).

Finally, since $\rho(x,x)=0$, for every $x\in I$, we must have $\psi(x)\equiv -\varphi(x)+c_0$ over $I$.
Replacing $h$ by $\tilde h(z)=h(z+c_0)$, $z\in I$,
the lemma follows (for $\tilde h$ and $\varphi$).
\end{proof}


We obtain the following analogue of Lemma~\ref{RSdZ}. 
\begin{theorem}\label{lem:quad}
Let $\gamma$, $\alpha$, and $\rho$ be as above. 
Then one of the following holds.\\
(i) For every finite set $A\subset (0,1)$ of size $n$,
$$
|\{(x,x',y,y')\in A^4\mid \rho(x,y)=\rho(x',y')\}|=O(n^{8/3}).
$$
(ii) There exists an open sub-interval $I\subset (0,1)$, such that, for every $x,y\in I$, 
$$\rho(x,y)=h(\varphi(x)-\varphi(y)),$$
where $\varphi,h$ are some univariate invertible analytic functions defined 
over $I$ and $J:=\varphi(I)-\varphi(I)$, respectively.
\end{theorem}

\begin{proof}
Suppose that $\gamma\subset \R^D$ is given by the system 
$$
g_i(x_1,\ldots,x_D)=0,~~i=1,\ldots, (D-1),
$$ 
where each $g_i$ is an irreducible constant-degree $D$-variate real polynomial.\footnote{Note that by
a curve $\gamma\subset\R^D$, we mean that $\gamma=\gamma_\C\cap \R^D$, where $\gamma_\C$ is a one-dimensional (irreducible) algebraic 
curve in $\C^D$ which is the 
zero set of a system of exactly $D-1$ polynomials of real coefficients.}
For every pair of points $\x=(x_1,\ldots,x_D), \y=(y_1,\ldots,y_D)\in \gamma$ of 
distance $\delta^{1/2}$, with $\delta\ge 0$, we have
\begin{align}\label{sysF}
g_i(\x)&=0,~~i=1,\ldots,(D-1),\\
g_i(\y)&=0,~~i=1,\ldots,(D-1),\nonumber\\
\|\x-\y\|^2-\delta&=(x_1-y_1)^2+\cdots+(x_D-y_D)^2-\delta=0.\nonumber
\end{align}

The system \eqref{sysF} defines a two-dimensional variety $V$ in $\R^{2D+1}$. 
Indeed, given a point $\x\in \gamma$ and a parameter $\delta\ge 0$, there exists at most $O(1)$ points
$\y\in\gamma$ such that $\|\x-\y\|^2=\delta$ (note that it is impossible for $\gamma$
to be contained in a sphere of radius $\delta^{1/2}$ centered at $\x$, since $\x\in\gamma$).
So locally $V$ can be described (analytically) by two parameters.

We apply a projection $\pi:\R^{2D+1}\to\R^3$ onto 
the coordinates $x_1,y_1,\delta$ of $\R^{2D+1}$.
By applying (in advance) a generic isometry in $\R^D$, we may assume that
the pre-image of each of the elements of $\pi(V)$ is finite. 
Indeed, applying such generic isometry, we may assume that $\gamma$ is not contained in a hyperplane of the form 
$\{\x=(x_1,\ldots,x_D)\in\R^D\mid x_1=a\}$, for some constant $a\in \R$.
So $\gamma$ intersects such hyperplane in at most $O(1)$ points.

By construction, we have
$$
Z_\alpha:=\{(x,y,\rho(x,y))\mid (x,y)\in(0,1)\}\subset \pi(V).
$$
Since $Z_\alpha$ is a graph of a bivariate analytic function (hence, forms a two-dimensional manifold), 
and it is contained in a two-dimensional algebraic variety (namely, the Zariski-closure of $\pi(V)$), it follows that
$Z_\alpha\subset Z(F)$, where $F$ is some irreducible trivariate real polynomial, and 
$Z(F)$ stands the zero set of $F$.
Note that $Z(F)\subset \pi(V)\cup Z_0'$, where $Z_0'$ is an algebraic variety in $\R^3$ which is at most one-dimensional.

Finally, we apply Lemma~\ref{RSdZ} to the polynomial $F$. 
Assume first that property (I) of Lemma~\ref{RSdZ} holds.
Then, for every $A\subset (0,1)$, with $|A|=n$, we have
\begin{equation}\label{Itoi}
|\{(a,a',b,b')\in A^4\mid \exists c\in\R.~ F(a,b,c)=F(a',b',c)=0\}|=O(n^{8/3}).
\end{equation}
By construction, $Z(F)$ identifies with the graph of the function $\rho$ over $(0,1)^2$. Thus
\eqref{Itoi} becomes
$$
|\{(a,a',b,b')\in A^4\mid \exists c\in\R.~ \rho(a,b)-c=\rho(a',b')-c=0\}|=O(n^{8/3}),
$$
or 
$$
|\{(a,a',b,b')\in A^4\mid \rho(a,b)=\rho(a',b')\}|=O(n^{8/3}),
$$
and so property (i) of Theorem~\ref{lem:quad} follows for this case.

Assume next that property (II) of Lemma~\ref{RSdZ} holds for the polynomial $F$.
Since $Z_0\cup Z_0'$ is at most one-dimensional, where $Z_0$ is the excluded set given in property (II),
there exists $v=(x_0,y_0,\delta_0)\in N\subset  Z_\alpha\setminus (Z_0\cup Z_0')$, where $N$ is some open neighborhood of $v$ in $Z_\alpha$.
By property (II), there exist open intervals $I_1,I_2\subset (0,1)$ containing $x_0$, $y_0$, respectively, 
and some neighborhood $I_3$ of $\delta_0$, 
and one-to-one real-analytic functions $\varphi_i: I_i \to \R$
with analytic inverses, for $i = 1, 2, 3$, such that
$$
(x,y,\delta)\in Z(F)~~\text{ if and only if}~~~ \delta =\varphi_3^{-1}(\varphi_1(x) + \varphi_2(y)),
$$
for every $(x, y, \delta) \in I_1\times I_2\times I_3$.
Equivalently, assuming that $I_1\times I_2\times I_3\subset N$ (by possibly shrinking them, if needed),
$$
\rho(x,y)=\delta~~\text{ if and only if}~~~ \delta=\varphi_3^{-1}(\varphi_1(x) + \varphi_2(y)),
$$
for every $(x,y,\delta)\in I_1\times I_2\times I_3$,
or
$$
\rho(x,y)=\varphi_3^{-1}(\varphi_1(x) + \varphi_2(y)),
$$
for every $(x,y)\in I_1\times I_2$.

Observe that in case that the last identity holds, we get
\begin{align*}
\det J_T&=\rho_1(x, y)\rho_2(x, y')\rho_2(x', y)\rho_1(x', y') -\rho_2(x, y)\rho_1(x, y')\rho_1(x', y)\rho_2(x',y')=0
\end{align*}
for every $(x,y)\in I_1\times I_1\times I_2\times I_2$.
Since $T$ is analytic, this implies that $\det J_T=0$ identically over $(0,1)^4$.
Applying Lemma~\ref{lem:special}, we get that property (ii) holds for this case.
This completes the proof of the lemma.
\end{proof}

\section{Distinct distances spanned by point sets lying on an algebraic curve}\label{sec:char}

\subsection{Proof of Theorem~\ref{charimp}}
Let $\gamma$ be an irreducible constant-degree algebraic curve in $\R^D$, 
and let $P$ be a set of $n$ points on $\gamma$.
Since $\gamma$ has constant degree, $\Omega(n)$ of the points of $P$ lie on 
some connected arc $\alpha\subset\gamma$, that has a 
parameterization of the form $\alpha(t)=(\alpha_1(t),\alpha_2(t),\ldots,\alpha_D(t))$, for $t\in (0,1)$,
where the $\alpha_i$ are analytic.
Thus, we may assume, without loss of generality, 
that $P\subset \alpha$. 
By applying (in advance) an isometry of $\R^D$, if needed, we may further assume that $\alpha_1(t)=t$
in this parameterization.
Letting $A:=\{t\in (0,1)\mid \alpha(t)\in P\}$, we get that $|A|=|P|=n$, 
and elements of $A$ correspond injectively to points of $P$.

Apply Theorem~\ref{lem:quad} to $\gamma,\alpha, \rho$, where
$\rho:(0,1)^2\to \R$ is defined as above. 
Then one of the properties (i) or (ii) in Theorem~\ref{lem:quad} holds.

Suppose first that property (i) holds. 
Let $\Delta$ denote the set of (squared) distances spanned by $P$.
We have
\begin{align*}
\binom{n}{2}
&=\sum_{\delta\in\Delta} \big|\big\{(x,y)\in A^2\mid \rho(x,y)=\delta\big\}\big|\\
&\le |\Delta|^{1/2}\left(\sum_{\delta\in\Delta} \big|\big\{(x,x',y,y')\in A^4\mid 
\rho(x,y)=\rho(x',y')=\delta\big\}\big|\right)^{1/2}\\
&\le |\Delta|^{1/2}\Big(\left|\left\{(x,x',y,y')\in A^4\mid \rho(x,y)=\rho(x',y')\right\}\right|\Big)^{1/2}\\
&= O\left(|\Delta|^{1/2}n^{4/3}\right), 
\end{align*}
where the inequality on the second line is due to the Cauchy-Schwarz inequality, and for the last 
line we used property (i). Rearranging, we get 
$$
|\Delta|=\Omega(n^{4/3}),$$
which completes the proof for this case.

Suppose next that property (ii) holds. Then there exists an open interval
$I\subset (0,1)$, such that 
$\rho(x,y)=h(\varphi(x)-\varphi(y)),$
for every $x,y\in I$,
where $\varphi,h$ are some univariate invertible analytic functions defined over $I$, $J:=\varphi(I)-\varphi(I)$, respectively.

Consider the transformation $S:I^3\to \R$ defined as
$$
S_3(x,y,z):= (\rho(x,y),\rho(y,z),\rho(x,z)).
$$
That is, $S$ maps a triple $(x,y,z)$, which is associated with a triple of points 
$p:=\alpha(x),q:=\alpha(y),r:=\alpha(z)$ on $\alpha$,
to the squared lengths of the edges of the triangle $pqr$ spanned by this triple.
It can be easily checked that the restrictive form of $\rho$, given by property (ii) in Theorem~\ref{lem:quad},  implies 
that $\det J_S=0$, for every $(x,y,z)\in I^3$.
Hence, for every $(a,b,c)=(\rho(x_0,y_0),\rho(y_0,z_0),\rho(x_0,z_0))$
in the image of $S$, the pre-image $S^{-1}(a,b,c)$ is at least one-dimensional, and can be interpreted as an (at least) 
one-dimensional family of triangles $pqr$ with vertices lying on $\gamma$ and with (squared) edge lengths $a,b,c$.

By Charalambides~\cite{Char1}, as reviewed for completeness in Subsection~\ref{charrev} below, this implies that $\gamma$ is an algebraic helix (specifically, see Corollary~\ref{cor:final} below).

\subsection{Flexible frameworks and algebraic helices}\label{charrev}
We provide a short review of (only the) relevant definitions and results from Charalambides~\cite{Char1}, which we need for the last step in our proof of Theorem~\ref{charimp}.\footnote{Note that some of the definitions and lemmas are given in \cite{Char1} in more generality; e.g., for a more general {\it distance function}.}
We start with the definition of a {\it flexible framework} on a smooth manifold $M\subset \R^d$. 
Informally, this is a vertex embedding of a graph into $M$ which can be moved continuously while preserving the length of each edge of the graph.
\begin{defi}[{Flexible framework~\cite[Definition~2.19]{Char1}}] 
Let $G=(V, E)$ be a graph and $M \subset\R^d$ be a smooth embedded submanifold of $\R^d$. Let $\phi:V\to M$ be an injective embedding of $V$ on $M$.  We say that $(G,\phi)$ is a {\em flexible framework} if there exists $$\Phi:V\times (-\delta,\delta)\to M,$$ for some $\delta>0$, such that,
$\Phi(\cdot,0)=\phi$, 
there exists $t_0\in (-\delta,\delta)$ such that $\Phi(\cdot,t_0)\neq \phi$, and, for each edge $\{u,v\}\in E$, the function 
$$t\mapsto\|\Phi(u,t)-\Phi(v,t)\|^2$$ 
is constant.
If $\Phi(v,\cdot)$ is smooth, for each $v\in V$,
we say that $(G,\phi)$ is {\it smoothly flexible} on $M$.
\end{defi}
 
\begin{defi}[{Degenerate curve \cite[Definition~2.23]{Char1}}]
Let $G$ be a graph. A smooth embedded curve
$\alpha\subset\R^d$ is called {\em $G$-degenerate} if for every embedding $\phi$ of $G$ on $\alpha$, the framework $(G,\phi)$ is smoothly flexible.
\end{defi}

Let $K_n$ denote the complete graph on $n$ vertices, for $n\ge 3$; note that $K_3$ is simply a triangle.
As it turns out,  the assumption that every triangle may be moved along a curve $\alpha$ while preserving the edge lengths implies that, in fact, any vertex embedding of any complete graph $K_n$ into $\alpha$ may be moved freely.

\begin{lemma}[{Chralambides~\cite[Lemma~7.1]{Char1}}]
Let $\alpha:I\to\R^d$ be a real analytic parameterization. Suppose that $\alpha$  is $K_3$-degenerate.
Then $\alpha$ is $K_n$-degenerate, for every $n\ge 1$.
\end{lemma}

This allows Charalambides to prove the following fact.
\begin{lemma}[{Chralambides~\cite[Lemma~7.2]{Char1}}]\label{Kn}
Let $\alpha:I\to\R^d$ be a real analytic parameterization. Suppose (as we may, without loss of generality), that $\alpha$ is a unit-speed parametrization.
Assume that $\alpha$  is $K_3$-degenerate.
Then, for each $k\ge 1$, the norm $\|\alpha^{(k)}\|$ of the $k$-th derivative of $\alpha$ is constant. 
\end{lemma}

One may then apply the following result of D'Angelo and Tyson~\cite{DAT}.
\begin{theorem}[{D'Angelo and Tyson~\cite[Corollary 3.8]{DAT}}]\label{DAT}
Let $\alpha:I\to\R^d$ be a real analytic parameterization. Suppose that, for each $k\ge 1$, the norm $\|\alpha^{(k)}\|$ of the $k$-th derivative of $\alpha$ is constant. 
Then there exist an orthogonal decomposition of the target space $\R^d = \R^{2m}\oplus \R^p$, an invertible skew-symmetric linear map $A$ on $R^{2m}$, vectors $v,v_0\in\R^{2m}$ and $w,w_0\in \R^p$, such that 
$$
\alpha(t)=\left(\left(\exp(At)-I\right)A^{-1}v+v_0,\;\;wt+w_0\right).
$$
\end{theorem} 

Combining Lemma~\ref{Kn} and Theorem~\ref{DAT}, one has the following.
\begin{corollary}\label{cor:genhel}
Let $\alpha:I\to\R^d$ be a real-analytic parameterization. Suppose (without loss of generality), that $\alpha$ is a unit-speed parameterization. Assume that $\alpha$  is $K_3$-degenerate. Then, up to a rigid motion,  
\begin{equation}\label{eq:genhel}
\alpha(t)=(\exp(At)v,tw,0)\in \R^{2k}\times\R^l\times\R^{2d-2k-l},
\end{equation}
where $v\in\R^{2k}$, $w\in\R^l$, and $2k+l\le d$.
\end{corollary}

Following Charalambides \cite[Definition 1.5]{Char1}, we call a curve of the form \eqref{eq:genhel} a {\it generalized helix}.

Recall that every $2k\times 2k$ invertible skew-symmetric matrix $B$ can be brought (see e.g. \cite{Youla}) to a block diagonal 
form $\Sigma$ by an orthogonal linear transformation $U$ of determinant 1, 
where $\Sigma$ is of the form
$$
\Sigma = 
\begin{bmatrix}
\begin{matrix}0 & \lambda_1\\ -\lambda_1 & 0\end{matrix} &  0 & \cdots & 0 \\
0 & \begin{matrix}0 & \lambda_2\\ -\lambda_2 & 0\end{matrix} &  & 0 \\
\vdots &  & \ddots & \vdots \\
0 & 0 & \cdots & \begin{matrix}0 & \lambda_k\\ -\lambda_k & 0\end{matrix} 
\end{bmatrix}
$$
for some real $\lambda_1,\ldots,\lambda_k$.
That is, $\Sigma = U^{-1}BU$, or $B= U\Sigma U^{-1}$. 
Simple manipulations then show that $\exp(B) = U\exp(\Sigma) U^{-1}$, and
$$
\exp(\Sigma) = 
\begin{bmatrix}
\begin{matrix}\cos\lambda_1 & \sin\lambda_1\\ -\sin\lambda_1 & \cos\lambda_1\end{matrix} &  0 & \cdots & 0 \\
0 & \begin{matrix}\cos\lambda_2 & \sin\lambda_2\\ -\sin\lambda_2 & \cos\lambda_2\end{matrix} &  & 0 \\
\vdots &  & \ddots & \vdots \\
0 & 0 & \cdots & \begin{matrix}\cos\lambda_k & \sin\lambda_k\\ -\sin\lambda_k & \cos\lambda_k\end{matrix} 
\end{bmatrix} .
$$
In other words, Corollary~\ref{cor:genhel} asserts that, for a suitable linear transformation $U$ of the
coordinate frame, we have
\begin{equation} \label{eqq0}
\alpha(t)=(r_1\cos\lambda_1t,r_1\sin\lambda_1t,\ldots,
r_k\cos\lambda_kt,r_k\sin\lambda_kt,tw,0),
\end{equation} 
for some $r_1,\ldots,r_k \in\R\setminus\{0\}$, $\lambda_1,\ldots,\lambda_k\in\R$.

In our setup, we only want to consider {\it algebraic} curves. 
The following lemma is due to Charalambides~\cite{Char1}.
\begin{lemma}[{Charalambides~\cite[Lemma 7.4]{Char1}}]\label{charlem}
Let $d>0$, $l, k \ge 0$ and $l + 2k = d$. Let $I\subset \R$ be an open interval.
Suppose that $\alpha:I \to\R^d$ is given by
$$
\alpha(t)=(r_1\cos\lambda_1t,r_1\sin\lambda_1t,\ldots,
r_k\cos\lambda_kt,r_k\sin\lambda_kt, tw)
$$
for some $r_1,\ldots,r_k, \lambda_1,\ldots,\lambda_k\in\R\setminus\{0\}$ and $w\in\R^l$.
Then $\alpha$ parametrizes an open subset of a real algebraic curve if and only if either $l = 0$ and
for each $1\le i,j\le k$ ratio $r_i/r_j$ is rational, or, alternatively, $k = 0$.
\end{lemma}
\noindent (Note that the implicit assumption $2k+l=d$ in the lemma involves no loss of generality.)

That is, a generalized helix is a real algebraic curve if and only if either $k = 0$ and $l> 0$ (in other words, it is a straight line) or, alternatively, $k > 0$, $l =0$ and has a parameterization of the form \eqref{eq:alghel}.

We thus conclude the following.
\begin{corollary}\label{cor:final}
Let $\gamma$ be an irreducible algebraic curve in $\R^d$. Let $\alpha:I\to\R^d$ be a real-analytic parameterization with $\alpha(I)\subset \gamma$. Suppose (as we may, without loss of generality), that $\alpha$ is a unit-speed parameterization. Assume that $\alpha$  is $K_3$-degenerate. Then $\gamma$ is an algebraic helix.
\end{corollary}

\section{Distinct distance subsets}\label{sec:subsets}
\subsection{Distinct distance subsets on algebraic curves}

For the proof of Theorem~\ref{main} we use the same inductive argument over the dimension $d$, used in \cite{CF} 
(the relevant theorem is cited here as Lemma~\ref{cf}). The new ingredient in our proof is the 
following bound on the quantity $h_{2,1,r}$, which is the case $d=1$ in Theorem~\ref{main}. 
This will form the base case for the induction, and will allow us to improve the general bound.
One can view Theorem~\ref{maincurve} as an extension of the result of 
Koml\'os, Sulyok, and Szemer\'edi~\cite{KSS} 
(see Lemma~\ref{kss}) to general algebraic curves, instead of the real line.

\begin{theorem}\label{maincurve}
For every $r\ge 1$, we have 
$$
h_{2,1,r}(n)=\Omega\left(n^{\frac4{9}}\right),
$$
where the constant of proportionality depends on $r$.
\end{theorem}

\begin{proof}
Let $\gamma$ be a constant-degree irreducible algebraic curve in $\R^D$, 
and let $P$ be a set of $n$ points on $\gamma$.
Since $\gamma$ has constant degree, $\Omega(n)$ of the points of $P$ lie in 
some connected arc $\alpha\subset\gamma$, that has a 
parameterization of the form $\alpha(t)=(t,\alpha_2(t),\ldots,\alpha_D(t))$, for $t\in (0,1)$.
Thus, we may assume, without loss of generality, 
that $P\subset \alpha$. 
Let $A:=\{t\in (0,1)\mid \alpha(t)\in P\}$. 
Then $|A|=|P|=n$, and elements of $A$ correspond injectively to points of $P$.

As in Section~\ref{sec:pre}, we define 
$$
\rho(x,y):=(x-y)^2+\sum_{i=2}^{D}(\alpha_i(x)-\alpha_i(y))^2,
$$
which is the squared distance between the two points on $\alpha$
parameterized by $x$ and $y$, and the transformation $T:(0,1)^4\to\R^4$, given by
$$
T(x,x',y,y'):=(\rho(x,y), \rho(x,y'),\rho(x',y),\rho(x',y')).
$$

Assume first that $\det J_T=0$ over $(0,1)^4$.
By Lemma~\ref{lem:special}, $\rho$ can be written as
 $\rho(x,y)=h(\varphi(x)-\varphi(y))$, for some univariate invertible 
analytic functions $h,\varphi$.
Applying Lemma~\ref{kss} to the image set $\varphi(A)$, and using the invertibility of $h$ and of $\varphi$,
we conclude that, in this case, there exists a subset $A'\subset A$ of size $\Omega(n^{1/2})$, such that 
all the nonzero values $\rho(x,y)$, with $x,y\in A'$, are distinct.

Assume next that $\det J_T$ is not identically zero over $(0,1)^4$.
By Theorem~\ref{lem:quad}, we have $|Q|=O(n^{8/3})$, where
$$
Q=Q(A):=\{(x,x',y,y')\in A^4\mid \rho(x,y)=\rho(x',y')\}.
$$
Let 
$$
S(A):=\{(x,y,y')\in A^3\mid \rho(x,y)=\rho(x,y')\}.
$$
Note that since $\gamma$ is irreducible and constant-degree, 
for every $p\in\gamma$, a circle centered at $p$,
for some point $p\in\gamma$, intersects $\gamma$ in at most $O(1)$ points, 
and thus contains $O(1)$ points of $P$.
Thus, $|S(A)|=O(n^2)$.

We now apply a probabilistic argument similar to the one used in 
\cite{Char2}.
We take a random subset $A_0\subset A$, such that each point $x$ of 
$A$ is chosen in $A_0$ independently, with probability $\pi$.
Let $Q(A_0)$ and $S(A_0)$ be as above.
We remove one point from each quadruple in $Q(A_0)$
and one point from each triple in $S(A_0)$, and let $A'\subset A_0$ be the resulting set.
Then, by construction, the distances spanned by $A'$ are pairwise distinct. 

We claim that for some choice of $A_0$, 
the set $A'$ is large enough. Indeed, 
$$
{\mathbb E}(|A'|)={\mathbb E}(|A_0|)-{\mathbb E}(|Q(A_0)|)-{\mathbb E}(|S(A_0)|)
$$
$$
{\mathbb E}(|A'|)\ge \pi n-\pi^4C_1n^{8/3}-\pi^3C_2n^{2},
$$
for some constants $C_1,C_2>0$.
Choosing $\pi\ge \frac{C}{n^{5/9}}$, with $C>0$ sufficiently small,
we get 
$$
{\mathbb E}(|A'|)\ge (C-C_1C^4) n^{4/9}-C_2C^3n^{1/3}=\Omega(n^{4/9}).
$$
This completes the proof of Theorem~\ref{maincurve}.
\end{proof}

\subsection{Proof of Theorem~\ref{main}}

Let $H_{a,d,r}(t)$ be the inverse function of $h_{a,d,r}(n)$. More precisely, $H_{a,d,r}(t)$ is the minimum $n$ 
such that, for any $d$-dimensional irreducible variety $V$ of degree $r$, any set of $n$ points lying on $V$ contains a subset of $t$ 
points for which all the non-zero $(a-1)$-dimensional volumes of the $\binom{t}{a}$ subsets of size $a$ are distinct.
Conlon, Fox, Gasarch, Harris, 
Ulrich, and Zbarsky~\cite{CF} proved the following relation.
\begin{lemma}[{\bf \cite[Theorem 4.2]{CF}}]\label{cf}
For all integers $r, d\ge 1$ and $a\ge 2$, there exist positive integers $r'=r'(a,d,r)$ and $C=C(a,d,r)$, such that,
for all integers $t\ge a$,
\begin{equation}\label{eq}
H_{a,d,r}(t)\le CH_{a,d-1,r'}(t)t^{2a-1}.
\end{equation}
\end{lemma}

By Lemma~\ref{cf}, applied with $a=2$, we have
$$
H_{2,d,r}(t)\le C H_{2,d-1,r'}(t)t^3.
$$
Iterating this recurrence relation $d-1$ times, we get
$$
H_{2,d,r}(t)\le \widetilde{C} H_{2,1,r}(t)t^{3(d-1)},
$$
for some constant $\widetilde{C}=\widetilde{C}(d,r)$.

By Theorem~\ref{maincurve}, we have
$$
H_{2,1,r}(t)=O(t^{9/4}).
$$
Combining these two inequalities Theorem~\ref{main} follows. \hfill\qed

\vspace{1cm}
\noindent{\bf Acknowledgment.}
I would like to thank Micha Sharir for helpful comments on a preliminary version of the paper. I would like to thank an anonymous referee for valuable comments and for explaining me how to simplify the reduction in Lemma~\ref{kss}.


\begin{thebibliography}{10}

\bibitem{BMP}
P. Brass, W. Moser, J. Pach, 
{\it Research Problems in Discrete Geometry}, 
Springer-Verlag, New York, 2005.



\bibitem{Char1}
M. Charalambides,
Distinct distances on curves via rigidity,
{\it Discrete Comput. Geom.} 51 (2014), 666--701.

\bibitem{Char2}
M. Charalambides, A note on distinct distance subsets,
{\it J. Geom.} 104 (2013), 439--442.

\bibitem{CF}
D. Conlon, J. Fox, W. Gasarch, D. G. Harris, 
D. Ulrich, and S. Zbarsky,
Distinct volume subsets, 
{\it SIAM J. Discrete Math.} 29 (2015), 472--480.


\bibitem{DAT}
J.P. D'Angelo and J.T. Tyson,
Helical CR structures and sub-Riemannian geodesics,
{\it Complex Var. Elliptic Equ.} 54 (2009), 205--221.

\bibitem{ER00}
G. Elekes and L. R\'onyai, 
A combinatorial problem on polynomials and rational functions,
{\it J. Combinat. Theory Ser. A} 89 (2000), 1--20.

\bibitem{ES12}
G. Elekes and E. Szab\'o, 
How to find groups? (And how to use them in Erd\H os geometry?), 
\emph{Combinatorica} 32 (2012), 537--571.

\bibitem{Erd}
P. Erd\H os, 
On sets of distances of $n$ points, 
{\it Amer. Math. Monthly} 53 (1946), 248--250.

\bibitem{ET}
P. Erd\H os and P. Tur\'an,
On a problem of Sidon in additive number theory, and on some related problems,
{\it J. London Math. Soc.} 1.4 (1941), 212--215.


\bibitem{GK2}
L. Guth and N. H. Katz,
On the Erd\H os distinct distances problem in the plane,
{\it  Annals Math.} 181 (2015), 155--190.

\bibitem{KSS}
J. Koml\'os, M. Sulyok and E. Szemer\'edi, 
Linear problems in combinatorial number theory,
{\it Acta Math. Acad. Sci. Hungar.} 26 (1975), 113--121.


\bibitem{PdZ}
J. Pach and F. de Zeeuw, 
Distinct distances on algebraic curves in the plane, 
{\it Combinat. Probab. Comput.} 26 (2017): 99--117.


\bibitem{RSdZ}
O. E. Raz, M. Sharir and F. de Zeeuw,
Polynomials vanishing on Cartesian products: The Elekes-Szab\'o Theorem revisited, 
{\it Duke Math. J.}, 165.18 (2016): 3517--3566.



\bibitem{SSS}
M. Sharir, A. Sheffer, and J. Solymosi,
Distinct distances on two lines,
{\it J. Combinat. Theory Ser. A} 120 (2013), 1732--1736.

\bibitem{Thi}
T. Thiele, 
{\it Geometric Selection Problems and Hypergraphs}, PhD thesis, Institut f\"ur Mathematik,
Freie Universit\"at Berlin, 1995.


\bibitem{Youla}
D. C. Youla,
A normal form for a matrix under the unitary congruence group,
{\it Can. J. Math.} 13  (1961), 694--704.

\end{thebibliography}
\end{document}